\documentclass[12pt,a4paper]{article}

\usepackage[utf8]{inputenc}
\usepackage[T1]{fontenc}
\usepackage{amsmath,amssymb,amsthm,amsfonts} 
\usepackage{mathtools} 
\mathtoolsset{showonlyrefs=true}
\usepackage[left=1.2in, right=1.2in, top=1in, bottom=1in]{geometry}
\usepackage{graphicx}  
\usepackage{subcaption}  
\usepackage{float}
\usepackage{hyperref,url}
\usepackage{cite}      
\usepackage{color}
\usepackage{xcolor}     
\usepackage{bm}        
\usepackage{authblk}
\usepackage[normalem]{ulem}
\usepackage{fullpage}
\usepackage[shortlabels]{enumitem}

\newtheorem{theorem}{Theorem}
\newtheorem*{theorem*}{Theorem}
\newtheorem*{proposition*}{Proposition}
\newtheorem{lemma}[theorem]{Lemma}
\newtheorem{corollary}[theorem]{Corollary}

\newtheorem{proposition}[theorem]{Proposition}
\newtheorem{remark}[theorem]{Remark}

\newcommand{\p}{\partial}
\newcommand{\R}{\mathbb{R}}
\newcommand{\N}{\mathbb{N}}
\newcommand{\C}{\mathbb{C}}
\renewcommand{\S}{\mathbb{S}}
\renewcommand{\sc}{\mathrm{sc}}
\newcommand{\inc}{\mathrm{inc}}
\newcommand{\eps}{\varepsilon}

\newcommand{\ii}{\mathrm{i}}
\DeclareMathOperator{\sinc}{sinc}
\newcommand{\supp}{\operatorname{supp}}
\newcommand{\comp}{\mathrm\comp}

\title{Reconstruction for an inverse scattering problem with a Kerr type nonlinearity}
\author[1]{Khaoula El Maddah}
\author[2]{Matti Lassas}
\author[3]{Tony Liimatainen}
\author[1]{Valter Pohjola}
\author[1]{Teemu Tyni}

\affil[1]{Research Unit of Applied and Computational Mathematics,
University of Oulu, Finland}
\affil[2]{Department of Mathematics and Statistics, University of Helsinki, Finland}
\affil[3]{Department of Mathematics and Statistics, University of Jyväskylä, Finland}
\date{}

\begin{document}

\maketitle

\begin{abstract}
We study the inverse scattering problem for the Kerr-nonlinear Helmholtz equation
\[
\Delta u + k^2(1+q(x)|u|^2)u = 0 \quad \text{in }\R^n,\; n\geq 2,
\]
where the aim is to recover the unknown potential $q$ from the scattering amplitude. We obtain uniqueness for full  data and partial data cases of  backscattering, fixed angle scattering, and fixed energy scattering.  For the linear Helmholtz equation, uniqueness in backscattering and fixed angle cases are classical and largely open problems. We are able to explicitly reconstruct individual Fourier modes of the potential, and if the measured directions and energies cover an open subset, we recover $q$.
The  simplicity of the approach leads to an efficient numerical method, and numerical experiments show accurate reconstructions, even in the presence of noise.

\end{abstract}


\section{Introduction}\label{sec_intro}

We study the inverse scattering problem for the Kerr-type nonlinear Helmholtz equation
\begin{equation}\label{eq_kerr_problem_intro}
 \Delta u + k^2\bigl(1+q(x)|u|^2\bigr)u = 0 \quad \text{in }\R^n,
\end{equation}
where $k>0$ is the wavenumber and $q$ is a compactly supported function describing the unknown potential. The total field is decomposed as $u = u_\inc + u_\sc$, where the incident field $u_\inc$ satisfies the homogeneous Helmholtz equation and the scattered field $u_{\sc}$ satisfies the Sommerfeld radiation condition
\begin{equation*}
\lim_{r\to\infty} r^{\frac{n-1}{2}}
\bigl(\partial_{r}-ik\bigr)u_{\sc}(x)=0,
\end{equation*}
where $r = |x|$.
For a sufficiently small incident wave, the scattering problem admits a unique radiating solution (see Section \ref{sec: direct_pro_frechet_diff}), and the scattered field has the unique asymptotic expansion
\begin{equation}\label{eq:farfield}
u_{\sc}(r,\hat \theta \,;\theta) = C_{k,n} \frac{e^{ikr}}{ r^\frac{n-1}{2}} A(k, \hat \theta,\theta) + o\bigl( r^\frac{-n+1}{2} \bigr), \quad r \to \infty,
\end{equation}
uniformly in $\hat \theta = x/|x|\in\S^{n-1}$, where $A(k,\hat\theta,\theta)$ is called the scattering amplitude. The classical inverse scattering problem is to recover the potential $q$ from knowledge of the scattering amplitude. 

The scattering amplitude depends on the incident direction $\theta$, the observation direction $\hat\theta$, and the wavenumber $k$ (and of course on the unknown potential $q$). We are particularly interested in partial data cases where only a subset of the full data is available, i.e., when $A(k,\hat\theta,\theta)$ is known only for $k\in U$, $\theta\in V$, and $\hat\theta\in W$, where $U\subset\mathbb{R}_+$ and $V,W\subset\S^{n-1}$. Important examples include backscattering ($\hat\theta=-\theta$), fixed angle scattering (either the incident direction $\theta$ or the observation direction $\hat\theta$ is fixed while the other varies), and fixed energy scattering ($k$ fixed).

Equation \eqref{eq_kerr_problem_intro} arises in nonlinear optics as a model for
monochromatic wave propagation in a Kerr medium, where the refractive index $n$
depends on the optical intensity as 
\[
n = n_0 + n_2 |E|^2.
\]
This intensity dependence leads to phenomena such as self-focusing of light beams.
We briefly outline the physical background here and refer to \cite{Boyd2008},
especially Section 4.1, for further details.

To see how the equation emerges, consider an electric field linearly polarized along the $x$-axis, $\mathbf E = (E,0,0)$, in a lossless, isotropic medium. The displacement field is $\mathbf D = \epsilon\mathbf E + \mathbf P^{\mathrm{NL}}$, where $\epsilon$ is a linear permittivity and $\mathbf P^{\mathrm{NL}} = (P^{\mathrm{NL}},0,0)$ is the nonlinear polarization, which gives the dipole moment density of the electric field, the Maxwell's equations give us the wave equation
\begin{equation}\label{eq:Maxwell_wave}
(-\Delta + \epsilon\partial_t^2)E = \partial_t^2 P^{\mathrm{NL}}.
\end{equation} 
The Kerr nonlinearity in equation \eqref{eq_kerr_problem_intro} arises from third order
interactions.
In centrosymmetric media such as silica glass, the nonlinear polarization satisfies the symmetry $P^{\mathrm{NL}}(-E) = -P^{\mathrm{NL}}(E)$. Expanding the nonlinear polarization as a power series in the electric field, the symmetry  forces all even-order terms to vanish. This leaves the third-order term as the dominant nonlinearity, so $P^{\mathrm{NL}} \approx \chi^{(3)} E^3$. 
We assume that $\chi^{(3)}$ is independent of time. For a monochromatic field $E(x,t) = \operatorname{Re}[E(x)e^{i\omega t}]= \frac{1}{2}(E e^{i\omega t} + \overline{E} e^{-i\omega t})$. Substituting this into \eqref{eq:Maxwell_wave} and equating the coefficients of $e^{i\omega t}$, the left side becomes $(\Delta + \epsilon\omega^2)E$. For the right side, expanding $E(x,t)^3$ produces terms at frequencies $e^{\pm i\omega t}$ and $e^{\pm 3i\omega t}$.
An intensity dependent refractive index is an effect that is not due to the third harmonic $e^{\pm 3i \omega t}$ terms, 
and the polarization $P^{NL}$ is modeled by the remaining 
$\frac{3}{8}\omega^2 \chi^{(3)}|E|^2 E e^{i \omega t}$ term.

Thus \eqref{eq:Maxwell_wave} yields
\[
(\Delta + \epsilon\omega^2)E = \frac{3}{8}\omega^2 \chi^{(3)} |E|^2 E,
\]
which is \eqref{eq_kerr_problem_intro} after setting $k^2 = \epsilon\omega^2$ and absorbing the constants and $\chi^{(3)}(x)$ into $q(x)$.

\subsection{Main results}

The main contributions of this work are twofold. First, by linearizing the scattering amplitude corresponding to the incident wave
\begin{equation*}
u_{\inc}(x) = \varepsilon e^{ik\theta\cdot x},
\qquad \theta\in\S^{n-1},\; \varepsilon>0,
\end{equation*}
to third order in the amplitude parameter $\varepsilon$,  we obtain an explicit relation for the Fourier transform $\mathcal{F}(q)$ of the unknown potential $q$:
\begin{equation}\label{eq:third_linearization_intro}
D_\varepsilon^3 A(k, \hat{\theta}, \theta;\eps)\big|_{\varepsilon=0} = 6\,\mathcal{F}(q)\bigl(k(\hat{\theta}-\theta)\bigr).
\end{equation}
This formula gives direct access to the Fourier data of $q$ in several physically relevant measurement configurations, including backscattering ($\hat\theta=-\theta$), fixed angle scattering, and fixed energy scattering ($k$ fixed). When this identity is known for all $k>0$ and all $\hat\theta\in\S^{n-1}$ with a fixed incident direction $\theta$, it yields a reconstruction formula for $q$ via Fourier inversion. The symmetric case with a fixed observation direction $\hat\theta$ and varying $\theta$ is analogous. Moreover, any open subset of frequencies suffices for the full recovery of $q$ by analytic continuation, meaning that the inverse problem is uniquely solvable under very mild measurement requirements.

Our main theoretical result is the following.

\begin{theorem}\label{thm: main theorem}
Let $q\in L^p(\R^n)$, with $p>\max\{2,n/2\}$, be compactly supported. Let $U\subset \R_+$ and $V,W\subset\mathbb{S}^{n-1}$. Define the set of accessible frequencies
\[
M \vcentcolon= \{\xi\vcentcolon=k(\hat \theta- \theta)\mid k\in U, \theta\in V, \hat \theta\in W\}\subset \R^n.
\]
Then the nonlinear scattering amplitude
\[
A(k,\hat\theta,\theta; \eps),\quad u_{\inc}=\eps e^{ik\theta\cdot x},
\]
for all sufficiently small $\eps>0$, all $k\in U$, $\theta\in V$, and $\hat \theta\in W$, uniquely determines $\mathcal{F}(q)(\xi)$ for all $\xi\in M$. If $M$ contains an open subset, then these scattering data uniquely determine $q$.
\end{theorem}

\begin{remark}
It suffices to consider $\varepsilon>0$. The scattering amplitude depends continuously on $\varepsilon$, and $A(k,\hat\theta,\theta;0)=0$. Moreover, equation \eqref{eq_kerr_problem_intro} is invariant under $u_\inc \mapsto -u_\inc$, since if $u$ is a solution corresponding to the incident field $u_\inc$, then $-u$ is a solution corresponding to $-u_\inc$. Hence the resulting scattering amplitude is unchanged.
\end{remark}

As a direct consequence, we obtain uniqueness in the physically important measurement configurations mentioned above.

\begin{corollary}\label{cor_main}
Under the assumptions of Theorem~\ref{thm: main theorem}, the mapping $u_{\inc}\mapsto A$ uniquely determines $q$ in the following cases:
\begin{enumerate}
    \item \textbf{Full data:} if $U=(0,\infty)$ and $V=W=\S^{n-1}$, then $M=\R^n\setminus\{0\}$.
    \item \textbf{Backscattering:} if $U=(k_1,k_2)$ for some $0\leq k_1<k_2 < \infty$ and $\hat \theta=-\theta\in \S^{n-1}$, then $M=\{\xi\in\R^n\mid 2k_1< |\xi|<2k_2\}$.
    \item \textbf{Fixed energy:} if $k=k_0\in (0,\infty)$ is fixed and $V=W=\S^{n-1}$, then $M=\{\xi\in\R^n\mid |\xi|\leq 2k_0\}$.
    \item \textbf{Fixed angle:} if $\theta=\theta_0\in \S^{n-1}$ (resp.\ $\hat \theta=\theta_0\in \S^{n-1}$) is fixed, and $U=(k_1,k_2)$ for some $0\leq k_1<k_2\leq \infty$, then $M$ is a (truncated) cone. Thus, if $U$ contains an open interval and $W$ (resp.\ $V$) has a relatively open subset on the sphere, then $M$ contains an open subset.
\end{enumerate}
\end{corollary}
We note that, in contrast to the linear Helmholtz equation, where these cases are either open in general or are proved through substantially deeper techniques related to the Calder\'on problem, our nonlinear approach yields a direct and elementary reconstruction in all these configurations. A more detailed discussion of the linear problem is given in the earlier works section below.

Second, we present a two-dimensional numerical framework for the reconstruction of $q$ from scattering data. The forward problem is solved via a Vainikko-type Lippmann--Schwinger formulation combined with Newton iteration. The Fourier data of $q$ is extracted using \eqref{eq:third_linearization_intro} from the third derivative of the scattering amplitude in the backscattering, fixed angle, and fixed energy measurement configurations. The inversion from this Fourier data is performed using Tikhonov regularization for smooth potentials and total variation regularization for discontinuous ones. The numerical experiments demonstrate that the approach produces accurate and stable reconstructions, even in the presence of noise.

\subsection{Earlier works}

Inverse scattering for the linear Helmholtz equation $\Delta u + k^2(1 + q(x))u = 0$ is a classical topic in applied mathematics. The problem of recovering the potential $q$ from far-field measurements has been studied extensively, and comprehensive treatments can be found in the monographs of Colton and Kress \cite{CK19}, Kirsch and Grinberg \cite{KirschGrinberg2008}, and Cakoni and Colton \cite{CakoniColton2014}.

For the linear equation, the status of the inverse problem depends heavily on the measurement configuration. For the fixed energy case, where all incident and observation directions are available at a fixed wavenumber $k$, uniqueness was proved by Nachman \cite{Nachman1992} and Novikov \cite{novikov} using methods that reduce the scattering problem to the Calder\'on problem. However, in partial data cases the situation is very different. The backscattering problem, where only $\hat\theta=-\theta$ is measured, is open in general and only partial results such as recovery of singularities are known \cite{GreenleafUhlmann,Ola,Paivarinta,SerovSandhu2010}. Similarly, the fixed angle scattering problem is largely open, with results again limited to the recovery of singularities \cite{SerovFixed}. For the time-dependent linear wave equation, recent progress on fixed angle scattering was made by Oksanen, Rakesh, and Salo \cite{Oksanen}.

For scattering problems with nonlinear models, one of the earliest uniqueness results was obtained by Jalade \cite{Jalade2004}, who considered a nonlinear Helmholtz equation in $\mathbb{R}^3$ and proved that the coefficient can be uniquely reconstructed from the scattering amplitude. For semilinear Schr\"odinger equations, which are closely related, inverse scattering was investigated by Serov and collaborators, see, e.g., \cite{SerovSandhu2010}. Nonlinear backscattering and fixed angle scattering were also studied in \cite{FotopoulosNonlinearFixedBack}, where recovery of singularities was proved.

The point of view that  nonlinearity of the equations may help in solving the inverse problems was
introduced in  \cite{KurLasUhl}. In this paper, 
it was shown that local measurements for the scalar wave equation with a quadratic nonlinearity determine the conformal class of a globally hyperbolic four-dimensional Lorentzian manifold. This result showed that nonlinearity can be exploited as a tool in inverse problems, as the analogous question for the linear wave equation remains open. 
The core idea is that higher-order Fr\'echet derivatives of the measurement operator with respect to the data encode information about the medium that is not accessible from the linearized problem alone.

Shortly after \cite{KurLasUhl}, the higher-order linearization method was adapted to elliptic equations. Feizmohammadi and Oksanen \cite{FeizmohammadiOksanen2020} and the related paper \cite{LLLS21} considered inverse problems for semilinear elliptic equations on Riemannian manifolds and in $\R^n$. 
Both works demonstrated that the method of \cite{KurLasUhl} extends well beyond the original hyperbolic setting. The partial data case, which is analogous to the limited measurement setups considered in the present paper, was subsequently addressed in \cite{krupchyk2020remark,lassas2020partial}. The method has also been applied to inverse scattering for nonlinear wave equations on Lorentzian manifolds \cite{Alexakis,Hintz,SaBarreto1,SaBarreto2}. These developments have attracted considerable attention and have established the higher-order linearization method as a widely used tool in the study of inverse problems for nonlinear equations. The present work can be seen as a counterpart of this line of research in the setting of inverse scattering in the frequency domain.  

Two contributions in inverse scattering are most closely related to the present work. The first is due to Furuya~\cite{Furuya2020}, who studied an inverse scattering problem for a semilinear Schr\"odinger equation with a general nonlinearity $a(x,u)$, assumed analytic in $u$, and proved that the scattering amplitude uniquely determines the nonlinearity. In that work, the wavenumber $k$ is fixed and the proof relies on complex geometric optics solutions together with a unique continuation argument to reduce the problem to an integral identity on a ball. Moreover, the recovery requires full scattering data, i.e., measurements for all incident and observation directions. In contrast, we vary $k$ to obtain a larger set of accessible frequencies, consider several limited data configurations such as backscattering and fixed angle scattering, and our reconstruction formula follows directly from the third-order linearization using only plane wave incident fields. No complex geometric optics solutions or unique continuation (in the full data case) are required, which is likely advantageous for stability and simplifies the numerical implementation.

The second is by Griesmaier, Kn\"oller, and Mandel \cite{GriesmaierKnoellerMandel2022}, who studied the same model as the present paper, namely the Kerr-type nonlinear Helmholtz equation $\Delta u + k^2(1 + q(x,|u|))u = 0$ with a generalized nonlinear refractive index. By linearizing the nonlinear far-field operator around zero and using the classical uniqueness result for the linear inverse medium scattering problem, they recursively determined the coefficients of the nonlinearity. They also generalized the factorization method and the monotonicity method to the nonlinear setting for shape reconstruction, and provided numerical examples. 
The factorization and monotonicity methods developed there are aimed at shape reconstruction, whereas our approach yields direct reconstruction of the coefficient $q$ itself.
Moreover, our uniqueness result requires only an open subset of accessible frequencies, which is a substantially weaker measurement condition.

\subsection{Organization
}
\medskip
\noindent
The paper is organized as follows. The main uniqueness results and the Fourier relation are proved in Section~\ref{sec: kerr_helmholtz}. Section~\ref{sec: direct_pro_frechet_diff} contains the well‑posedness and Fr\'echet differentiability  of the forward problem. The numerical reconstruction framework is described in Section~\ref{sec: numerics_frame}, with examples in Section~\ref{sec: numerical_examples}.

\section{Proof of the main results} \label{sec: kerr_helmholtz}

In this section, we present the proof of Theorem~\ref{thm: main theorem}
and briefly discuss its consequences. 
Let us begin by restating the problem introduced in section \ref{sec_intro}. We consider
the nonlinear scattering problem for the scalar Helmholtz equation with a Kerr-type contrast $q$,

\begin{equation}\label{eq_kerr_problem}
\begin{cases}
\Delta u + k^2\bigl(1+q(x)|u|^2\bigr)u = 0,
& x\in \mathbb{R}^n, \\[2mm]
u = u_\inc + u_\sc, \\[2mm]
u_{\inc}(x)=\varepsilon e^{ik\theta\cdot x}, 
& \theta\in\mathbb{S}^{n-1},\ k>0, \\[2mm]
\displaystyle
\lim_{|x|\to\infty} |x|^{\frac{n-1}{2}}
\left(\frac x{|x|}\cdot \nabla-ik\right)u_\sc(x)=0.
\end{cases}
\end{equation}
Here $u_{\inc}$ denotes the incident plane wave of amplitude $\varepsilon>0$,
and $u_\sc$ is the scattered field satisfying the Sommerfeld radiation condition.

Let $q\in L^p_\mathrm{comp}(\mathbb{R}^n)$ with
 $p>\max\{2,n/2\}$.
Then, for solutions $u$ with $u_\sc$ in an appropriate weighted Sobolev space $H^2_{-\delta}(\R^n)$ (see definition \eqref{eq_weighted_Sobolev} and Proposition~\ref{prop_well_posed_and_Frechet_Rn}),
problem \eqref{eq_kerr_problem} can be reformulated equivalently
as the nonlinear Lippmann-Schwinger integral equation
\begin{equation}\label{eq_lippmann_schwinger}
u(x)=u_\inc(x)
+\int_{\mathbb{R}^n} G_k(|x-y|)\,k^2 q(y)\,|u(y)|^2u(y)\,dy,
\end{equation}
where $G_k$ denotes the outgoing radiating fundamental solution  of the Helmholtz operator
$-\Delta-k^2$, see 	\eqref{eq_Green_a}. Using the asymptotics of the Green's function
and \eqref{eq_lippmann_schwinger}, one can derive the asymptotic expansion
$$
u_{\sc}(r,\hat \theta \,;\theta) = C_{k,n} \frac{e^{ikr}}{ r^\frac{n-1}{2}} A(k, \hat \theta,\theta;\eps) + o\big( r^\frac{-n+1}{2} \big), \quad \text{ as } r \to \infty,
$$
where we used polar coordinates $x=(r,\hat\theta)$, and where $A(k,\hat \theta,\theta;\eps)$ is called the scattering amplitude, given by
\begin{equation*}
A(k, \hat \theta,\theta;\eps)=\int_{\R^n}e^{-ik \hat \theta\cdot y}q(y)|u(y)|^2u(y)dy.
\end{equation*}

\medskip

We will use linearization to analyze the differential equation in \eqref{eq_kerr_problem}.
Taking the first derivative in $\eps$ of the Helmholtz equation in \eqref{eq_kerr_problem}
gives us
$$
\Delta u_\eps + k^2 u_\eps + k^2q(x)(u^2\overline{u_\eps}+2|u|^2u_\eps) = 0,
$$
where $u_\eps := D_\eps u$ is the Fr\'echet derivative.
It follows that the first linearization $v:=D_\eps u|_{\eps=0}$ of $u$ with respect to $\eps$ at $\eps=0$
satisfies
$$
\Delta v + k^2v = 0, \quad \text{ in } \R^n,
$$
which follows from the previous equation, since $u=0$, when $\eps = 0$. Due to the Sommerfeld radiation condition and the Rellich uniqueness theorem (see Lemma 2.12 in \cite{CK19})
we have that $v_{sc}=0$, and thus $v=e^{ik\theta\cdot x}$. By a similar computation one can work out the second 
linearization $D^2_\eps u |_{\eps=0}$.

Throughout this paper, we adopt the Fourier transform
\[
\mathcal F (u)(\xi)
:= \int_{\mathbb R^n} u(x)e^{-i x\cdot \xi}\,dx
\]
defined on the Schwartz class $\mathcal S(\R^n)$.
For this normalization, the inverse Fourier transform is given by
\[
\mathcal F^{-1}u(x)
=
(2\pi)^{-n}
\int_{\mathbb R^n} u(\xi)e^{i x\cdot \xi}\,d\xi.
\]
The Fourier transform extends by duality to the space of tempered distributions \(\mathcal S'(\mathbb R^n)\), and we use the same notation for the extension.

Now we consider the derivatives of the scattering amplitude in $\eps$ (see Proposition~\ref{prop_frechet_for_A}).
Since the nonlinearity is cubic, the first nontrivial contribution
to the scattering amplitude arises at third order. More precisely,
by a straightforward computation, which again uses the fact that $u|_{\eps=0}=0$,
one obtains
\begin{equation}
D_\eps^3 A(k, \hat{\theta}, \theta;\eps)\big|_{\eps=0}
=
6\int_{\R^n}e^{-ik\hat{\theta}\cdot y} q(y)|v|^2v \,dy, \quad v(x)=e^{ik\theta\cdot x}.
\end{equation}
Substituting the explicit expression for $v$, we compute
\begin{equation} \label{third_linarization_amp}
\begin{aligned}
    D_\eps^3 A(k, \hat{\theta}, \theta;\eps)\big|_{\eps=0} &=
6\int_{\R^n} q(y) e^{-ik(\hat{\theta} - \theta)\cdot y} dy\\
&= 6\, \mathcal F(q)(k(\hat{\theta}-\theta)).
\end{aligned}
\end{equation}
Consequently, the third-order linearization of the scattering amplitude
provides direct access to the Fourier data of the potential function $q$. We
can prove that the scattering amplitude related to \eqref{thm: main theorem}
uniquely determines $q$, if the set of accessible frequencies has an interior.

\begin{proof}[Proof of Theorem~\ref{thm: main theorem}]
The third linearization \eqref{third_linarization_amp} is justified by Proposition~\ref{prop_frechet_for_A}. Thus, 
\[
D_\eps^3 A(k,\hat \theta,\theta;\eps)\big|_{\eps=0} = 6\mathcal{F}(q)(\xi)\quad\text{for }\xi=k(\hat \theta-\theta)\in M.
\]
Next, suppose $M$ contains an open set $O\subset M$. Since $q\in L^p(\R^n)$
defines a compactly supported distribution then, due to the Paley-Wiener-Schwartz theorem, $\mathcal{F}(q)$ is real-analytic,
see~\cite[Theorem~7.3.1]{Hormander1}. Analytic continuation of $\mathcal{F}(q)$ from
$O$ to $\R^n$  determines $\mathcal{F}(q)$ uniquely, whence $q$ is uniquely determined.
\end{proof}

\noindent
It should be noted that the condition $p>\max\{2,n/2\}$ is only imposed to justify
the forward problem and the third-order linearization formula
\eqref{third_linarization_amp}. However, the actual uniqueness argument in
Theorem~\ref{thm: main theorem} only uses this formula together with the
compact support of $q$, which implies via the Paley--Wiener--Schwartz theorem
that $\mathcal{F}(q)$ is real-analytic (indeed, an entire function).

Finally, we note that the claims in Corollary \ref{cor_main} are obtained directly from Theorem \ref{thm: main theorem}.

\section{The direct problem and Fr\'echet differentiability}\label{sec: direct_pro_frechet_diff}
In this section we prove that the scattering problem \eqref{eq_kerr_problem} admits a unique solution
for sufficiently small incident waves. We then study the Fr\'echet differentiability of the
corresponding solution operator and of the scattering amplitude. We accommodate low $L^p$ integrability and allow singularities for the potential. For this purpose we adapt some methods from linear theory to the current nonlinear setting.

We consider sufficiently small incident waves $u_{\inc}$ satisfying
$(\Delta+k^2)u_{\inc}=0$. Writing the total field as
$u=u_{\inc}+u_{\sc}$, the scattered field $u_{\sc}$ satisfies

\begin{equation}\label{eq_scattering_problem_a}
\begin{cases}
\Delta u_{\sc} + k^2 u_{\sc} + k^2 q(x)\,|u(x)|^2 u(x) = 0, & \text{in } \R^n,\\[2mm]
\displaystyle \lim_{|x|\to\infty} |x|^{\frac{n-1}{2}}\big(\frac x{|x|}\cdot \nabla-ik\big)u_{\sc}(x)=0,
\end{cases}
\end{equation}
where the limit holds uniformly.
Recall that we assume that $q\in L^p(\R^n)$ with $p>\max\{2,n/2\}$ and that $\operatorname{supp} q \subset B := B(0,R)$, 
for some $R>0$.  
The scattered solution $u_\sc$ can be found by solving the Lippmann-Schwinger equation \eqref{eq_lippmann_schwinger},
see Proposition \ref{prop_LS_to_u_sc} below.

The Green's function $G_k$ of the Helmholtz equation, that appears in \eqref{eq_lippmann_schwinger} admits the 
explicit representation
\begin{equation}\label{eq_Green_a}
G_k(x)
= k^{n-2}\frac{i}{4}\Big(\frac{1}{2\pi k|x|}\Big)^{\frac{n-2}{2}}
H^{(1)}_{\frac{n-2}{2}}(|k||x|),
\qquad x\in\R^n\setminus\{0\},
\end{equation}
where $H^{(1)}_{\nu}$ is the Hankel function of the first kind; see, for instance,
\cite[Theorem 9.4 and (9.14)]{McLean00}.
In particular, in dimension $n=3$ one has
\[
G_k(x)=\frac{e^{ik|x|}}{4\pi|x|} ,
\qquad x\in\R^n\setminus\{0\}.
\]
For completeness, we also record the one-dimensional case. When $n=1$, the outgoing fundamental solution
of $\partial_x^2+k^2$ is given explicitly by
\[
G_k(|x|)=\frac{i}{2k}e^{ik|x|},
\]
and the corresponding volume potential has particularly simple mapping properties. In this dimension it
is sufficient to assume $q\in L^1(\R)$ (together with $\operatorname{supp}(q)\subset B$) 
in order to set up the Lippmann-Schwinger formulation and prove well-posedness for small incident amplitudes.

We define the weighted Sobolev spaces $H^m_\delta(\R^n)$, $m\in \N_0$ and $\delta \in \R$ as
\begin{equation} \label{eq_weighted_Sobolev}
\begin{aligned}
H^m_\delta(\R^n) := \{ u \in L^2_\mathrm{loc}(\R^n) \,:\, \langle x  \rangle^\delta D^\alpha u  \in L^2(\R^n), |\alpha| \leq m \},
\end{aligned}
\end{equation}
where $ \langle  x \rangle = (1+|x|^2)^{1/2}$, $D = (-i \p)^\alpha$, and $\alpha$ is a multi-index. 
We also use the notation $L^2_\delta(\R^n) := H^0_\delta(\R^n)$.
The proof of the following is essentially the same as in the linear case (see \cite{CK19}). We omit the proof. 

\begin{proposition}\label{prop_LS_to_u_sc}
Let $q\in L^p(B)$ with $p>\max\{2,n/2\}$ and $u$ be a solution of the Lippmann-Schwinger equation \eqref{eq_lippmann_schwinger}, define
$u_{\sc}:=u-u_{\inc}$. Then $u_{\sc}\in L^\infty(B)\cap H^s_{-\delta}(\R^n)$, $\delta>\frac12$, satisfies \eqref{eq_scattering_problem_a}, i.e.
\[
\Delta u_{\sc}+k^2u_{\sc}+k^2 q|u|^2u=0 \quad \text{in }\R^n,
\]
and $u_{\sc}$ fulfills the Sommerfeld radiation condition.
\end{proposition}

\noindent
Next, we define the convolution operator $\mathcal G_k$ by
\begin{equation}\label{eq_G_conv_def}
(\mathcal G_k v)(x)
:= \int_{\R^n} G_k(x-y)\, q(y)\, v(y)\,dy,
\qquad x\in\R^n.
\end{equation}
We have the following estimates for $\mathcal G_k$.

\begin{lemma} \label{lem_G_bounded}
Suppose $q \in L^p(B)$, $p > \max\{1,n/2\}$. Then 
$$
\| \mathcal G_k v \|_{L^\infty(B)} \leq C_k \| v \|_{L^\infty(B)}.
$$
\end{lemma}

\begin{proof}

Let $x\in B$. Since $\operatorname{supp}(q)\subset B$, we have
\[
(\mathcal G_k v)(x)=\int_B G_k(x-y)\,q(y)\,v(y)\,dy,
\]
so that,
\begin{equation}\label{eq:Gk_basic_bound}
|(\mathcal G_k v)(x)| \le \|v\|_{L^\infty(B)} \int_B |G_k(x-y)|\,|q(y)|\,dy.
\end{equation}
Let us now specify the behaviour of $G_k$ at zero. 
The Hankel function appearing in formula \eqref{eq_Green_a} behaves as
\begin{equation}\label{eq_H1_asymp}
H^{(1)}_{\frac{n-2}{2}}(|k||x|)
\sim
\begin{cases}
-(i/\pi) C_n |kx|^{(2-n)/2}, &  n\geq 3, \\
 \;(2i/\pi ) \log(|kx|),   &  n=2.
\end{cases}
\qquad
\text{ as } |kx| \to 0,
\end{equation}
where $C_n$ is given by the Gamma function, see formulas (10.7.2) and (10.7.7) in \cite{DLMF}. 
We also know that $H^{(1)}_{\frac{n-2}{2}}(x)$ is analytic on $x>0$, and bounded when $x \to \infty$ 
(see formula (10.7.8) in\cite{DLMF}). 
It follows from formula \eqref{eq_Green_a}, using \eqref{eq_H1_asymp} and the power series expansions
of Bessel functions (see (10.4.3), (10.2.2), (10.2.3) and (10.8.2) in \cite{DLMF}) that 
\begin{equation}\label{eq_Gk_decomp}
G_k(x)= k^{n-2}
\begin{cases}
C_n |z|^{2-n} + h(z), & \quad n\geq 3, \\
C_2 \log|z| + h(z),   & \quad n = 2,
\end{cases},
\qquad \quad z := |kx|,
\end{equation}
where $h$ is an analytic function on $z>0$ depending on $n$, 
that behaves at worst as $h \sim |z|^{2-n+\epsilon}$ at zero, for some $\epsilon > 0$,
in the case $n\geq 3$, and where $h$ is bounded in the case $n=2$.
The asymptotic behaviour of $G_k$ at zero coincide with that of the kernel of the Newtonian potential $\Gamma$, where $\Gamma$ is 
given by  
\begin{align*}
\Gamma(x-y) = 
\begin{cases}
 \int_B  |x-y|^{2-n}dy, &  n\geq 3, \\
 \int_B | \log(|x-y|)|dy,   &  n=2.
\end{cases}
\end{align*}
For $\Gamma$  we have by Sobolev embedding (see e.g. p.158 in \cite{GT01}) and Theorem 9.9 in \cite{GT01} that for all $n/2 < p < \infty$ 
\begin{align*}
\|\Gamma * |q|\|_{L^\infty(B)} \leq  \|\Gamma * |q| \|_{W^{2,p}(B)} \leq  \|q \|_{L^{p}(B)}. 
\end{align*}
Since we can write $h(z) = \Gamma(z) |z|^{n-2} h(z)$, and $|z|^{n-2} h(z) \in L^\infty(B)$, we also
get by the above that
\begin{align*}
 \big\| |h| * |q| \big\|_{L^\infty(B)}
\leq  
 \big\| |z|^{n-2} h(z)  \big\|_{L^\infty(B_{2R})}  \big\| \Gamma * |q| \big\|_{L^\infty(B)}
\leq  
C  \big\|\Gamma * |q|  \big\|_{W^{2,p}(B)} \leq  \|q \|_{L^{p}(B)}, 
\end{align*}
with $B_{2R} := B(0, 2R)$,
in the case $n \geq 3$.
Thus by \eqref{eq_Gk_decomp}  we have 
\begin{align*}
|\mathcal G_k v|  
\leq 
C_k  \| v \|_{L^\infty(B)} 
(\big\|\Gamma * |q|\big\|_{L^\infty(B)} 
+ \big\| h * |q|\big\|_{L^\infty(B)})
\leq 
C_k  \| v \|_{L^\infty(B)}  \|q \|_{L^{p}(B)}. 
\end{align*}
\end{proof}

\noindent
We introduce the closed ball in $C(\overline B)$ (equipped with the sup norm)
\[
\mathcal B_\delta := \{u\in C(\overline B): \|u\|_{L^\infty(B)} \le \delta\}.
\]

\begin{lemma}\label{lem_cubic_nemytskii}
Let $B\subset \mathbb{R}^n$ be measurable. The
operator
\[
\mathcal N : L^\infty(B) \to L^\infty(B), \qquad \mathcal N(u)=|u|^2u,
\]
is $C^\infty$ (real) Fr\'echet differentiable.
\end{lemma}
\begin{proof}
We view $L^\infty(B)=L^\infty(B;\mathbb{C})$ as a real Banach space.
Writing $\mathcal N(u)=u\cdot \overline{u}\cdot u$, the map is a composition of the continuous bilinear map $(u,v)\mapsto uv$ and the real linear map $u\mapsto \overline{u}$.
Finite compositions of continuous multilinear maps between Banach spaces are smooth in the real Fr\'echet sense (see Theorems 2.2.1 and 2.4.3 in \cite{Ca83}).
Hence $\mathcal N$ is real $C^\infty$  Fr\'echet differentiable.
\end{proof}

Proposition~\ref{prop_well_posed_and_Frechet} gives the scattered field $u_{\sc}$ restricted to $B$.
In Proposition~\ref{prop_well_posed_and_Frechet_Rn} we then extend $u_{\sc}$ to all of $\R^n$.

\begin{proposition} \label{prop_well_posed_and_Frechet}
The following holds:
\begin{enumerate}[(a)]
\item
There exists $\eps_k >0$ such that equation \eqref{eq_lippmann_schwinger} has a unique solution $u=u_\inc + u_\sc$ with $u\in \mathcal B_{\eps_k} $,  when $u_\inc \in \mathcal B_{\eps_k}$ and  $0 < \eps < \eps_k$.

\item Let $S: \mathcal B_{\eps_k} \to L^\infty(B)$, be the map $S:u_\inc \mapsto u_\sc$. Then $S$ is three times Fr\'echet differentiable.
\end{enumerate}
\end{proposition}

\begin{proof}
We apply the implicit function theorem \cite[Theorem~10.6]{RR04}.  
Define
\[
F: L^\infty(B)\times L^\infty(B)\to L^\infty(B),\qquad
F(u,u_{\inc}) := u - u_{\inc} - \mathcal G_k\big(|u|^2u\big).
\]
By Lemma~\ref{lem_G_bounded}, the operator $\mathcal G_k$ maps $L^\infty(B)$ to $L^\infty(B)$, hence $F$ is well-defined and \(F(0,0)=0\). 

To verify the hypotheses of the implicit function theorem, we next show that
\begin{enumerate}[(i)]
\item $F$ is $C^3$ Fr\'echet differentiable in a neighbourhood of $(0,0)$;
\item the partial derivative $D_uF(0,0):L^\infty(B)\to L^\infty(B)$ is a bounded isomorphism.
\end{enumerate}
Define the nonlinear map
\[
\mathcal N:L^\infty(B)\to L^\infty(B),\qquad
\mathcal N(u)=|u|^2u
\]
By Lemma~\ref{lem_cubic_nemytskii} $\mathcal N$ is $C^3$ on $L^\infty(B)$. Since \(\mathcal G_k:L^\infty(B)\to L^\infty(B)\) is linear and
bounded by Lemma~\ref{lem_G_bounded}, the composition
\(\mathcal G_k\circ\mathcal N\) is real \(C^3\). Hence \(F\) is real \(C^3\), proving (i).
In particular, at $(u,u_{\inc})=(0,0)$ we obtain $D_uF(0,0)[h]=h$, so $D_uF(0,0)=\mathrm{Id}$ is a bounded bijection on
$L^\infty(B)$, thus $(ii)$ also holds.

The equation $F(u,u_{\mathrm{inc}})=0$ is equivalent to the Lippmann–Schwinger equation
\eqref{eq_lippmann_schwinger}. By the implicit function theorem~\cite[Theorem 10.6]{RR04}, there exists $\varepsilon_k>0$
and a unique map
\[
S:\mathcal B_{\varepsilon_k}\subset L^\infty(B)\to L^\infty(B)
\]
such that $u=S(u_{\mathrm{inc}})$ solves $F(u,u_{\mathrm{inc}})=0$ for every
$u_{\mathrm{inc}}\in\mathcal B_{\varepsilon_k}$. This proves (a). Moreover, due to the implicit function theorem and~\cite[Remark~10.5]{RR04}, the solution map $S$ is $C^3$, which yields (b).
\end{proof}

Proposition~\ref{prop_well_posed_and_Frechet} yields a solution $u_{\sc}$ of the Lippmann-Schwinger equation
restricted to the ball $B$. We now show that $u_{\sc}$ admits a natural extension to all of $\R^n$. 

\begin{proposition}\label{prop_well_posed_and_Frechet_Rn}

Assume $q\in L^p(\R^n)$ with $p>\max\{2,n/2\}$, and define $\widetilde u_\sc(x)$ for $x \in \R^n$ by
\begin{equation}\label{eq:uscattered}
\begin{aligned}
\widetilde u_\sc(x) = k^2 \int_{\R^n} G_k(x-y) q(y)  |u(y)|^2 u(y)  \, dy,
\end{aligned}
\end{equation}
where $u := (u_\inc+S(u_\inc))(y) \in L^\infty(B)$ and $S$ is the operator in Proposition \ref{prop_well_posed_and_Frechet}, and $u_\inc \in \mathcal B_{\varepsilon_k}$.
Then $\widetilde u_\sc \in H_{-\delta}^2(\R^n)$, for $\delta > 1/2$, and the operator  $u_\inc  \mapsto \widetilde u_\sc:  L^\infty(B)\to H^2_{-\delta}(\R^n)$ is three times Fréchet differentiable.
\end{proposition}

\begin{proof}

Fix $\delta>1/2$ and let $u_{\inc}\in\mathcal B_{\varepsilon_k}$ be as in Proposition~\ref{prop_well_posed_and_Frechet}.
Set
\[
f := k^2\,q\,|u|^2u.
\]
Recall that
$
f:=k^2\,q\,|u|^2u,
$
and 
$
u=u_{\inc}+u_{\sc},
$
so that $\supp(f)\subset\supp(q)\subset B$. 
To apply Agmon's estimates, we use the fact that $p>2$, so that by H\"older's inequality $q\in L^2(B)$, which for $u\in L^\infty(B)$, after extension by zero, gives
$f\in L^2(B)\subset L^{2}_\delta(\R^n)$ and
\begin{equation}\label{eq:f_bound}
\|f\|_{L^{2}_\delta(\R^n)} \le C\,k^2\,\|q\|_{L^2(B)}\,\|u\|_{L^\infty(B)}^3.
\end{equation}
where $C$ depends on $B$ and $\delta$, because $(1 + |x|^2)^{\delta/2}$ is bounded on $B$.

Let $k>k_0>0$. Denote by
\[
R_0^+(k^2):=\lim_{\varepsilon\to 0^+}(-\Delta-(k^2+i\varepsilon))^{-1}
\]
the outgoing resolvent. Equivalently, $R_0^+(k^2)$ is the integral operator with the kernel 
$G_k$ \cite[Ch.~23, eq. (23.1)]{serov2017fourier}, so that
$$
\mathcal G_k = R_0^+(k^2) (q \,\cdot).
$$
By Agmon's weighted resolvent estimate together with the limiting absorption principle~\cite[Theorem~A.1 and Theorem~4.1]{Ag75}, there exists a constant $C$ such that
\begin{equation}\label{eq:agmon_est}
\|R_0^+(k^2)f\|_{H^2_{-\delta}(\R^n)} \le C k \,\|f\|_{L^{2}_\delta(\R^n)},
\qquad f\in L^{2}_\delta(\R^n).
\end{equation}
Since $\widetilde u_{\sc} = \mathcal G_k (k^2 |u|^2u) = R_0^+(k^2)f $ we get
from \eqref{eq:agmon_est} and \eqref{eq:f_bound} that $\widetilde u_{\sc}\in H^2_{-\delta}(\R^n)$ and 
\begin{equation}\label{eq:final_estimate}
\|\widetilde u_{\sc}\|_{H^2_{-\delta}(\R^n)}
\leq
C\,k^3\,\|q\|_{L^2(B)}\,\|u\|_{L^\infty(B)}^3.
\end{equation}
This proves the first part of the claim.

Next we prove the Fr\'echet differentiability of the map $u_{\inc}\mapsto \widetilde u_{\sc}$. We have $\widetilde u_{\sc} = \mathcal G_k(k^2 \mathcal N(u_{\inc} +  S(u_\inc))$. 
By Proposition~\ref{prop_well_posed_and_Frechet}, the map $S:\mathcal B_{\varepsilon_k}\to L^\infty(B)$ is $C^3$ and by Lemma~\ref{lem_cubic_nemytskii} $\mathcal N:L^\infty(B)\to L^\infty(B)$ is real $C^\infty$. We have that $\mathcal G_k :
L^\infty(B) \to L^\infty(B)\subset L^{2}_\delta(\R^n)$ (after extension by zero) is linear and bounded by Lemma~\ref{lem_G_bounded} therefore real $C^3$.
Therefore the composition $u_{\inc}\mapsto \widetilde u_{\sc}= k^2 \mathcal G_k\circ \mathcal N(u_{\inc} + \mathcal S (u_\inc))$ is three times Fr\'echet
differentiable from $\mathcal B_{\varepsilon_k}\subset L^\infty(B)$ into $H^2_{-\delta}(\R^n)$.
\end{proof}

\noindent

\begin{proposition}\label{prop:usc_unique_LS}
The function $u:=u_\inc+\widetilde u_\sc$, where the extension $\widetilde u_{\sc}:\R^n\to\C$ is provided by Proposition \ref{prop_well_posed_and_Frechet_Rn}, is the unique solution of the scattering problem \eqref{eq_kerr_problem}. 
\end{proposition}

\begin{proof}
Let $f:=k^2 q|u|^2u$. 
Since $G_k$ is the outgoing fundamental solution of $\Delta+k^2$, it satisfies the
Sommerfeld radiation condition. Hence, the volume potential
\[
\widetilde u_{\sc}(x)=(G_k*f)(x)=\int_{\R^n}G_k(x-y)\,f(y)\,dy
\]
is a radiating solution of $(\Delta+k^2)\widetilde u_{\sc}=-f$, see, \cite[Ch.~9]{McLean00}.

To prove uniqueness, let $\widetilde u_{\sc,1}$ and $\widetilde u_{\sc,2}$ be two radiating fields such that $u_i = \widetilde u_{\sc, i} + u_\inc$ 
satisfies \eqref{eq_kerr_problem} and define $w:=u_1-u_2$. Then, for $x\in\R^n\setminus B$ we have $q(x)=0$, and since both $u_1$ and $u_2$ solve
$(\Delta+k^2)v=0$ in $\R^n\setminus\overline B$ the function $w$ satisfies
\[
(\Delta+k^2)w=0\quad\text{in }\R^n\setminus\overline B,
\]
and $w$ is radiating (as the difference of two radiating fields). By Rellich's lemma \cite{CK19}, it follows that
\[
w=0\quad\text{in }\R^n\setminus\overline B.
\]
Hence $\widetilde u_{\sc,1}=\widetilde u_{\sc,2}$ in $\R^n\setminus\overline B$. Proposition \ref{prop_well_posed_and_Frechet} shows that $\widetilde u_{\sc,1}(x) = \widetilde u_{\sc,2}(x)$ in $B$, so that $w=0$ in $\R^n$, which proves uniqueness.
\end{proof}

\begin{proposition} \label{prop_frechet_for_A}
Let $q\in L^p(\R^n)$, $p>\max\{2,n/2\}$, with compact support, 
and let $k>0$, and $ \hat \theta,\theta\in \S^{n-1}$. Suppose  moreover that 
$u_{\inc} = \eps e^{ik \theta \cdot x}$.
Then there exists $\eps_k>0$ such that the mapping
$\eps \mapsto A(k, \hat \theta,\theta;\eps)$, with $\eps \in (-\eps_k,\eps_k)$, is three times
differentiable.
\end{proposition}
\begin{proof}
Recall that
\[
A(k,\hat\theta,\theta;\varepsilon)
=
\int_{\R^n}
e^{-ik\hat\theta\cdot y}
q(y)
|u(y,k,\theta;\varepsilon)|^2
u(y,k,\theta;\varepsilon)
\,dy.
\]
Since $q$ has compact support, the integral is effectively taken over a bounded set
$B\supset \operatorname{supp}(q)$.
By Proposition \ref{prop_well_posed_and_Frechet_Rn}, there exists
$\eps_k>0$ such that the map
\[
\eps \mapsto u(\cdot,k,\theta;\eps)
\]
is three times Fr\'echet differentiable as a map
$(-\eps_k,\eps_k)\to L^\infty(B)$. From Lemma~\ref{lem_cubic_nemytskii} it follows that
\[
\eps \mapsto |u(\cdot,k,\theta;\eps)|^2u(\cdot,k,\theta;\eps)
\]
is also three times real-Fr\'echet differentiable as an $L^\infty(B)$-valued map. In particular, $\partial_\eps^j\bigl(|u|^2u\bigr)\in L^\infty(B)$ for $j=0,1,2,3.$

Finally, note that the map $\mathcal A:L^\infty(B)\to\C$ defined by
\[
f\mapsto \int_{B}e^{-ik\hat\theta\cdot y} q(y)f(y)dy
\]
is a bounded linear functional, hence $C^\infty$ in the Fr\'echet sense, implying that
\[
A(k,\hat\theta,\theta;\eps)=\mathcal A(|u|^2u)
\]
is three times differentiable in the parameter $\eps\in(-\eps_k,\eps_k)$.
\end{proof}

\section{Numerical reconstruction for the nonlinear Helmholtz equation}\label{sec: numerics_frame}
In this section, we present a numerical framework for the reconstruction of the potential $q$ from nonlinear Helmholtz scattering data. In section \ref{subsec: limmmann-schwinger}, we reformulate the direct problem as a Lippmann-Schwinger integral equation and solve it numerically using the Vainikko method combined with Newton's method to handle the nonlinearity. In section \ref{subsec: F(q)} we derive a formula that relates the Fourier transform of the potential $q$ to the scattering measurement; in particular, we focus on backscattering, fixed angle, and fixed energy cases, and in section \ref{subsec: rec of q from F(q)} we reconstruct $q$
from its Fourier transform using either Tikhonov or TV regularization. 

We consider the nonlinear Helmholtz equation in two dimensions
\begin{equation}
\Delta u(x) + k^2 \bigl(1 + q(x)\,|u(x)|^2\bigr) u(x) = 0,
\qquad x \in \mathbb{R}^2,
\end{equation}
where $k > 0$ is the wavenumber and $q$ is a compactly supported nonlinear potential on $B$.
The total field $u$ is decomposed as
\begin{equation*}
u(x) = u_\inc(x) + u_\sc(x),
\end{equation*}
where the incident wave is a plane wave
\begin{equation*}
u_\inc(x) = \varepsilon e^{ik \theta \cdot x},
\qquad \theta \in \mathbb{S}^1,
\end{equation*}
and the scattered field $u_\sc$ satisfies the Sommerfeld radiation condition.

\subsection{Lippmann-Schwinger formulation}\label{subsec: limmmann-schwinger}
Using the Green's function of the Helmholtz operator, the nonlinear Helmholtz equation can be written as the nonlinear
Lippmann-Schwinger integral equation
\begin{equation}
u(x) = u_\inc(x)
+ \int_{\mathbb{R}^2}  G_k(x-y)\, k^2 q(y)\, |u(y)|^2 u(y)\, dy,
\end{equation}
where in the two dimensional case
the radiating fundamental solution~\eqref{eq_Green_a}
of the Helmholtz operator is
\[
G_k(x-y)=\frac{i}{4}H_0^{(1)}(k|x-y|).
\]
Since \(q\) is supported in the ball $B$, the domain of integration reduces to $B$.

To solve the Lippmann-Schwinger equation numerically, we use the Fast Fourier Transform (FFT)-based
approach introduced by Vainikko~\cite{vainikko2000fast}. In this approach, the convolution operator appearing in the Lippmann-Schwinger equation is approximated on a uniform grid. The Fourier coefficients of the Helmholtz fundamental solution are used to represent the convolution kernel, which allows the convolution to be evaluated efficiently using FFTs.

We define the nonlinear operator \(\mathcal{T}: L^\infty(B_{2R}) \to L^\infty(B_{2R})\) by
\begin{equation}
\mathcal{T}(u)(x) = u(x)-u_\inc(x)-\int_{B}  G_k(x-y)\, k^2 q(y)\, |u(y)|^2 u(y)\, dy,
\end{equation}
reducing the Lippmann-Schwinger equation to locating $u\in L^\infty(B_{2R})$ such that
\begin{equation*}
\mathcal{T}(u) = 0.
\end{equation*}

The nonlinear system $\mathcal{T}(u)=0$ is solved using Newton's method. We will solve the system iteratively as
\begin{equation}
    u^{n+1} = u^n + \delta u^n,\quad u_0=u_\inc,
\end{equation}
where the Newton correction $\delta u^n$ satisfies
\begin{equation}\label{eq: correction term}
\mathcal{T}'(u^{n})[\delta u^n]=-\mathcal{T}(u^{n}).
\end{equation}
Since the Fr\'echet derivative of the nonlinearity $f(u) = |u|^2 u$
is
\begin{equation}
f'(u)[v] = 2|u|^2 v + u^2 \overline{v},
\end{equation}
the equation \eqref{eq: correction term} is given explicitly as the real-linearized system
\begin{equation}
\delta u^n(x) - \int_{B}  G_k(x-y)\, k^2 q(y)
\left( 2|u^n(y)|^2 \delta u^n(y) + (u^n(y))^2 \overline{\delta u^n(y)} \right) dy = -\mathcal{T}(u^n)(x).
\end{equation}
The equation can be rewritten in terms of the Fourier transform as
\begin{equation}\label{eq: correction delta u}
\delta u^n(x) - \mathcal{F}^{-1}\!\left(
\mathcal{F}(G_k)(\xi)\,
\mathcal{F}(\,k^2 q \left( 2|u^n|^2 \delta u^n + (u^n)^2 \overline{\delta u^n} \right))(\xi)
\right)(x) = -\mathcal{T}(u^n)(x),
\end{equation}
where, after discretization, the Fourier coefficients of $G_k$ are computed explicitly following the formulas on page~16 of \cite{vainikko2000fast}.

After discretization on the computational grid, this equation yields a real-linear system for the unknown correction \(\delta u^n\). By separating real and imaginary parts, it is converted into a real-valued linear system, which is solved using the GMRES method. More details on the discretization are given in the next paragraph.

\paragraph{Discretization of the forward problem}
Following the approach of Vainikko, the computational domain is restricted to a square
\[
\Omega = [-L,L]^2,
\qquad L = \alpha R,
\]
containing the support of the potential.
The domain is then discretized using a uniform Cartesian grid of size $2N \times 2N = 2^M \times 2^M$ points, where $M\in\mathbb{N}$ denotes the refinement level. The mesh size and the discrete grid points are given respectively by 

\[\qquad h = \frac{2L}{2N-1} , \qquad
x_{ij} = (-L + ih,\,-L + jh),
\qquad 0\leq i,j \leq 2^M-1.
\]

In order to use the GMRES method, we use the real-system reformulation similarly as in~\cite{eirola2003solution,huhtanen2012numerical}.
Denoting the (complex-valued) discretized linearized equation \eqref{eq: correction delta u} by
\begin{equation}
B[u^n](\delta u^n) = Y, 
\qquad \delta u^n \in \mathbb{C}^N, \quad Y \in \mathbb{C}^N,
\end{equation}
we convert the complex system into a real one
$\mathbb{C}^N \cong \mathbb{R}^{2N}$. 

Let
\[
\delta u^n = V_R + i V_I, 
\qquad 
V_R, V_I \in \mathbb{R}^N, \quad \text{and} \quad
V = \begin{pmatrix} V_R \\ V_I \end{pmatrix} \in \mathbb{R}^{2N}.
\]
Similarly, we define
\[
Y = Y_R + i Y_I \qquad 
Y_R, Y_I \in \mathbb{R}^N, \quad \text{and} \quad \widetilde Y = \begin{pmatrix} Y_R \\ Y_I \end{pmatrix} \in \mathbb{R}^{2N}.
\]
Finally, define the real-linear operator 
$\widetilde B [u^n]: \mathbb{R}^{2N} \to \mathbb{R}^{2N}$ by
\[
\widetilde B[u^n](V)
=
\begin{pmatrix}
\operatorname{Re}\big( B[u^n](V_R + i V_I) \big) \\
\operatorname{Im}\big( B[u^n](V_R + i V_I) \big)
\end{pmatrix}.
\]
The original complex system is thus equivalent to the real system
\begin{equation}
\widetilde B[u^n](V) = \widetilde Y,
\end{equation}
which is linear over $\mathbb{R}$ and has dimension $2N$. This allows the use of real GMRES in a matrix-free way, since only the action of the operator $\widetilde B[u^n]$ on a vector is required and can be performed via FFT, and the corresponding $2N\times 2N$ matrix is never assembled explicitly.

The action of the Jacobian operator is evaluated via
numerical quadrature without explicitly forming the 
dense matrix. This reduces memory usage and allows 
efficient solution of large-scale
nonlinear scattering problems.

\subsection{Formulation of Fourier transform of the potential $q$}\label{subsec: F(q)}
The scattering amplitude
 is defined by
\begin{equation}
A(k, \hat{\theta}, \theta;\eps)
=
\int_{\R^n}
e^{-ik \hat{\theta}\cdot y}
q(y)\, |u(y)|^2 u(y)\, dy.
\end{equation}
Once the field $u$ has been computed on the discretization grid, this integral is approximated by a discrete sum over the grid points containing the support of $q$ using the uniform rectangular quadrature rule associated with the Vainikko discretization.
The third derivative of $A(k, \hat{\theta}, \theta;\eps)$ with respect to $\varepsilon$ at $\varepsilon=0$ is calculated using a Savitzky-Golay filter (i.e., fitting a local polynomial to the data in a neighborhood of $\varepsilon=0$ and analytically differentiating the fitted polynomial \cite{savitzky1964smoothing}).

Let \(k_\ell \in [0,K_{\max}]\) denote the radial coordinate, and let
\(\phi_j,\varphi_j \in [0,2\pi)\) denote angular coordinates. Furthermore,
\(\xi_{r,x}\) and \(\xi_{r,y}\) denote the corresponding Cartesian components of the Fourier frequency vector. We choose \(N_k\) radii and \(N_\theta\) angles, resulting in a total of \(N_kN_\theta\) sampling points forming a polar grid in the Fourier domain. We consider three limited data measurement configurations. Based on the equation \eqref{third_linarization_amp}, we have:

\paragraph{Backscattering:}
In the case of backscattering $\hat\theta=-\theta$, and the third linearization gives
\begin{equation}
    D_\eps^3 A(k, -\theta, \theta;\eps)\big|_{\eps=0} = 6\,\mathcal F(q)(-2k\theta).
\end{equation}
The sampled frequencies are
\[
\xi = -2k\theta,
\]
and for isotropic resolution, we sample the frequency points \( \xi \) in polar coordinates, that is,
\begin{equation}
   \xi_r = -2  k_\ell \theta_j = -2  k_\ell ( \cos\varphi_j, \;  \sin\varphi_j) = (\xi_{r,x}, \xi_{r,y}).
\end{equation}

\paragraph{Fixed angle scattering, $\hat \theta$ fixed:} We fix the observation direction $\hat\theta= \hat \theta_0$. Then
\begin{equation}
    D_\eps^3 A(k, \hat \theta_0, \theta;\eps)\big|_{\eps=0} = 6\,\mathcal F(q)\bigl(k(\hat\theta_0-\theta)\bigr).
\end{equation}
The frequencies are given by
\[
\xi = k(\hat{\theta}_0-\theta),
\]
and we sample the frequency points \( \xi \) in polar coordinates as
\begin{equation}
   \xi_r = k_\ell (\hat \theta_0 - \theta_j) = k_\ell( \hat \theta_{0,1} - \cos\varphi_j, \; \hat \theta_{0,2} - \sin\varphi_j) = (\xi_{r,x}, \xi_{r,y}), 
\end{equation}

\paragraph{Fixed energy scattering, $k$ is fixed:} When the frequency $k = k_0$ is fixed, one obtains 
\begin{equation}
    D_\eps^3 A(k_0, \hat \theta, \theta;\eps)\big|_{\eps=0} = 6 \,\mathcal F(q)\bigl(k_0(\hat\theta-\theta)\bigr).
\end{equation}
Here the frequency points
\[
\xi = k_0(\hat{\theta}-\theta),
\]
cover a closed disk $\overline{B_{2k_0}}$, and we sample the frequency points \( \xi \) in polar coordinates
\begin{equation}
   \xi_r = k_0 (\hat \theta_j - \theta_j) = k_0 (\cos\phi_j - \cos\varphi_j, \; \sin\phi_j - \sin\varphi_j) = (\xi_{r,x}, \xi_{r,y}), 
\end{equation}

\subsection{Reconstruction of the potential $q$ from its Fourier data}\label{subsec: rec of q from F(q)}
Let $\Omega \subset \mathbb{R}^2$ be partitioned into rectangular cells 
$p_{m,n}$ with
\begin{equation*}
(m,n)\in\{0,\dots,N_x-1\}\times\{0,\dots,N_y-1\}. 
\end{equation*}

We approximate a complex-valued function 
$g : \Omega \to \mathbb{C}$ by a piecewise constant expansion
\begin{equation}
g(X) = \sum_{m,n} g_{m,n}\,\varphi_{m,n}(X),
\end{equation}
where the basis functions $\varphi_{m,n}$ are defined as
\begin{equation}
\varphi_{m,n}(X) =
\begin{cases}
1, & X \in p_{m,n}, \\
0, & \text{otherwise}.
\end{cases}
\end{equation}

To compute the Fourier transform of \( g \), we evaluate the contribution of each cell \( p_{m,n} \) at a discrete set of frequencies \( \xi_r  \in \mathbb{R}^2 \), resulting in the expression
\[
\widehat{g}(\xi_r)=\int_{p_{m,n}} g(X)\, \mathrm{e}^{ -\ii \xi_r \cdot X} \, \mathrm{d}X 
= \sum_{m,n}g_{m,n} \int_{p_{m,n}} \mathrm{e}^{ -\ii \xi_r \cdot X} \, \mathrm{d}X.
\]

Let
\[
c_{m,x} = x_m + \frac{h}{2},
\qquad
c_{n,y} = y_n + \frac{h}{2}
\]
denote the cell center.

We next build the matrix \( E \in \mathbb{C}^{A \times N} \), where $N = N_x  N_y$ and $A = N_k  N_\theta$, and whose entries \( E_{(A_x, A_y),(m,n)} \) represent the contribution of the cell \( p_{m,n} \) to the Fourier coefficient associated with frequency \( \xi_r \), so the entries of Fourier matrix
\begin{equation*}
\begin{aligned}
E_{(A_x, A_y),(m,n)}=h^2\sinc(\frac12\xi_{r,x}h)\sinc(\frac12\xi_{r,y}h)e^{-i(\xi_{r,x}c_{m,x} + \xi_{r,y}c_{n,y})}
\end{aligned}
\end{equation*}
In matrix form:
\begin{equation} \label{eq:to inverse}
  \mathbf d = E \mathbf q,  
\end{equation}
where \(\mathbf q\in\mathbb R^{N}\) denotes the vector of discretized values of the potential $q$ on the reconstruction grid, and \(\mathbf d\) denotes the vector of Fourier samples obtained from the third-order linearization of the scattering amplitude.

A regularized solution to \eqref{eq:to inverse} is then found using Tikhonov or TV regularization (Chambolle-Pock primal-dual algorithm) depending on the smoothness of the potential.

\begin{remark}
The choice of the Fourier matrix formulation is motivated by the flexibility required to treat different measurement configurations, such as backscattering, fixed-angle, and fixed-energy data, as well as situations where the number of Fourier samples differs from the number of reconstruction unknowns. The matrix formulation naturally accommodates limited and irregular frequency sampling patterns and allows regularization techniques, such as Tikhonov or TV regularization, to be incorporated directly into the inversion procedure. While FFT-based methods typically rely on Cartesian frequency grids, nonuniform FFT (NUFFT) techniques could also be considered for irregular sampling geometries.
\end{remark}

\section{Numerical examples}\label{sec: numerical_examples}

All numerical experiments are performed in two spatial dimensions on the square domain
\[
\Omega = [-L,L]^2, \qquad L = 2\,R,
\]
where $R$ denotes the width of a box containing the support of $q$.
The domain is discretized using a uniform Cartesian grid. In the numerical experiments below we take $M=8$ (so that $N=2^{M-1}=128$) as described in ~\ref{subsec: limmmann-schwinger}.

To consider measurement uncertainty, we corrupt the scattering amplitude by additive complex Gaussian noise and define
\[
A_{\mathrm{noisy}}(k,\hat \theta,\theta;\eps)
= A(k,\hat \theta,\theta;\eps) + \sigma\big(\eta_1 + i\,\eta_2\big),
\]
where $\eta_1,\eta_2\sim \mathcal N(0,1)$ are independent standard normal random variables.
The noise amplitude is chosen proportional to the maximal signal magnitude,
\[
\sigma = \rho \,\max_{\varepsilon_k}\big|A(k,\hat \theta,\theta;\eps)\big|,
\qquad 
u_{\inc}(x)=\varepsilon_k e^{ik\theta\cdot x},
\]
where $\rho>0$ is the prescribed relative noise level. In all examples below we set $\rho=0.02$, corresponding to
$2\%$ relative complex Gaussian noise, which yields an average signal-to-noise ratio of approximately $23$~dB.

The third derivative of the amplitude with respect to the perturbation parameter \( \varepsilon \) is approximated using a Savitzky–Golay polynomial filter with \( N_\varepsilon = 40 \) sampling points uniformly distributed in \( \varepsilon \in [-0.05, 0.05 ] \) and a smoothing window of size \(25\). 

Fourier measurements are sampled on a polar grid, as described in Section~\ref{subsec: F(q)}. In all experiments, we use \(N_r = 40\) radial samples and \(N_\theta = 40\) angular directions. In the backscattering setting, the maximum radius is chosen as $R_{\max} = 15$. In the fixed-angle setting, we take $R_{\max} = 30$ and
\[
\theta_0 = 
\frac{(5,1)}{\sqrt{26}}\in\mathbb{S}^1,\quad
\]
while in the fixed-energy setting, we set \(k_0 = 15\). 

The regularization parameters for Tikhonov and TV regularization were chosen experimentally and are specified for each example below.

\paragraph{Example 1.}
For the backscattering case, we consider the potential $q_1$ defined as the sum of two smooth, compactly supported
elliptic bumps with radius $R=0.4$. The first bump is centered at $(0.4,0.4)$ with orientation angle $\alpha=\frac{\pi}{4}$. The second bump is centered at
$(-0.5,-0.4)$ with orientation
$\alpha=\frac{\pi}{2}$, and amplitude times $\frac{2}{3}$.

With this potential, we use Tikhonov regularization and take the 
regularization parameter $\lambda = 10^{-4}$.

The results are shown in Figure~\ref{fig:two-bump}, and reflect a good reconstruction of the two smooth ellipses, even with different amplitude.

\paragraph{Example 2.}
In the fixed-angle scattering configuration, we consider a potential supported in a kite-shaped domain. The shape is then rotated by an angle \(\alpha\). 

The potential is defined as the indicator function of the interior of this rotated kite-shaped domain \(D\), namely
\[
q_2(x, y)=
\begin{cases}
1 & \text{if } (x,y) \in D\\[1mm]
0 & \text{otherwise},
\end{cases}
\]
 In the numerical experiment, we take \(R=0.4\), and \(\alpha=\frac{\pi}{4}\).

In Figure~\ref{fig:kitte-shape}, the reconstruction obtained with TV regularization and with $\lambda = 10^{-2}$, closely captures the kite-shape.

 \paragraph{Example 3.}
For the fixed-energy scattering case, we consider a \(K\)-shaped potential of the form
\[
q_3(x, y)=\begin{cases}
1 & \text{if } (x,y) \in K\\[1mm]
0 & \text{otherwise},
\end{cases} 
\]
where \(K\subset\mathbb R^2\) is a rotated and translated K-shaped domain.

Figure~\ref{fig:k-shape} demonstrates a good reconstruction of the K-shaped potential with $\lambda = 10^{-1}$ for TV regularization.

 \paragraph{Example 4.} We also consider a flower-shaped potential in the case of fixed-energy scattering. The support is described in local polar coordinates \((\rho,\theta)\) by
\[
W=
\left\{
\beta R \le \rho \le R\bigl(1+a\cos(5\theta)\bigr)
\right\},
\]
with \(a=0.28\) and \(\beta=0.35\). 

\[
q_4(x, y)=
\begin{cases}
1 & \text{if } (x,y) \in W\\[1mm]
0 & \text{otherwise},
\end{cases}
\]

Thus, the outer boundary is a five-lobed star-shaped curve, while the interior contains a circular hole of radius \(0.35R\) with $R=0.6$.

Figure~\ref{fig:flower-shape} represents the result for the reconstruction of the flower-shape, for this example, we use TV regularization with $\lambda = 10^{-1}$, which captures closely the sharp contour.

\begin{figure}[H]
  \centering
  \includegraphics[width=0.95\textwidth]{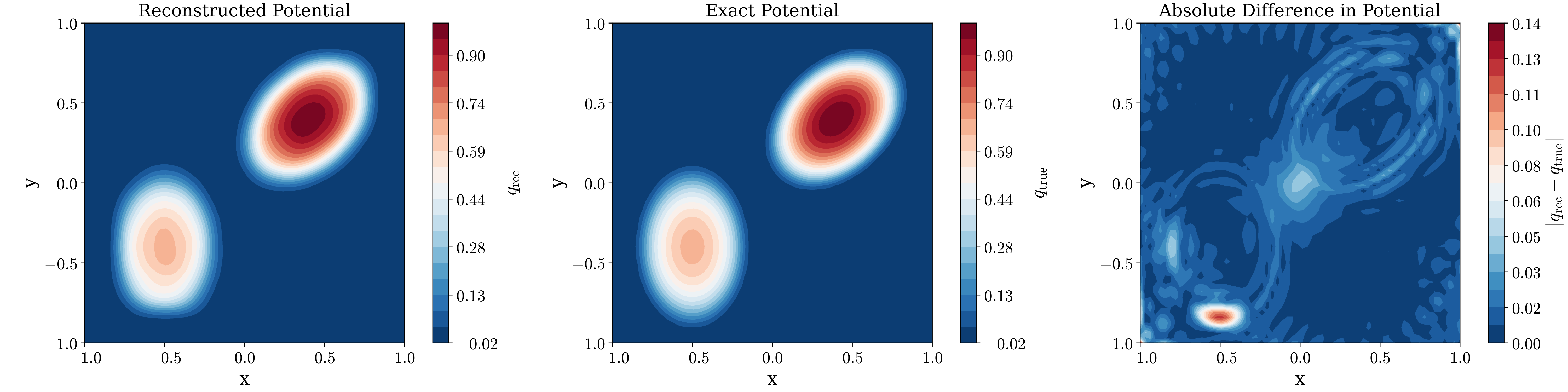}
  \caption{Example 1, backscattering case with a potential $q$ given by two elliptic bumps: reconstruction (left), ground truth (middle), and residual (right). The position and magnitude of the ellipses are well reconstructed, with an $L^2$ error of $0.057$.}
  \label{fig:two-bump}
\end{figure}

\begin{figure}[H]
  \centering
  \includegraphics[width=0.95\textwidth]{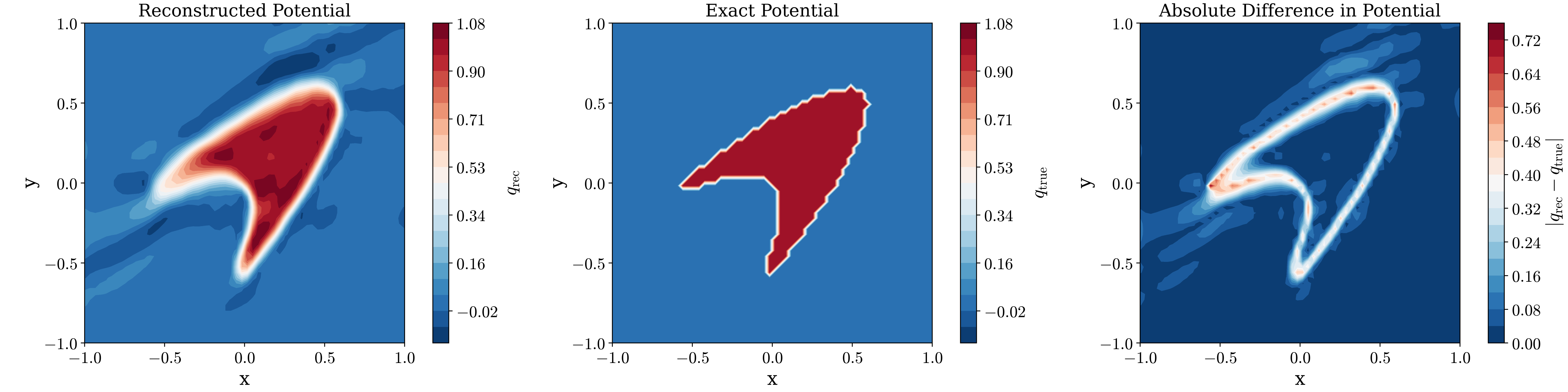}
  \caption{Example 2, kite-shape potential for the fixed-angle scattering case: reconstruction (left), ground truth (middle), and residual (right). The reconstruction presents some artifacts near the jump discontinuities, but the overall kite shape is well recovered , with an $L^2$ error of $0.298$.}
  \label{fig:kitte-shape}
\end{figure}

\begin{figure}[H]
  \centering
  \includegraphics[width=0.95\textwidth]{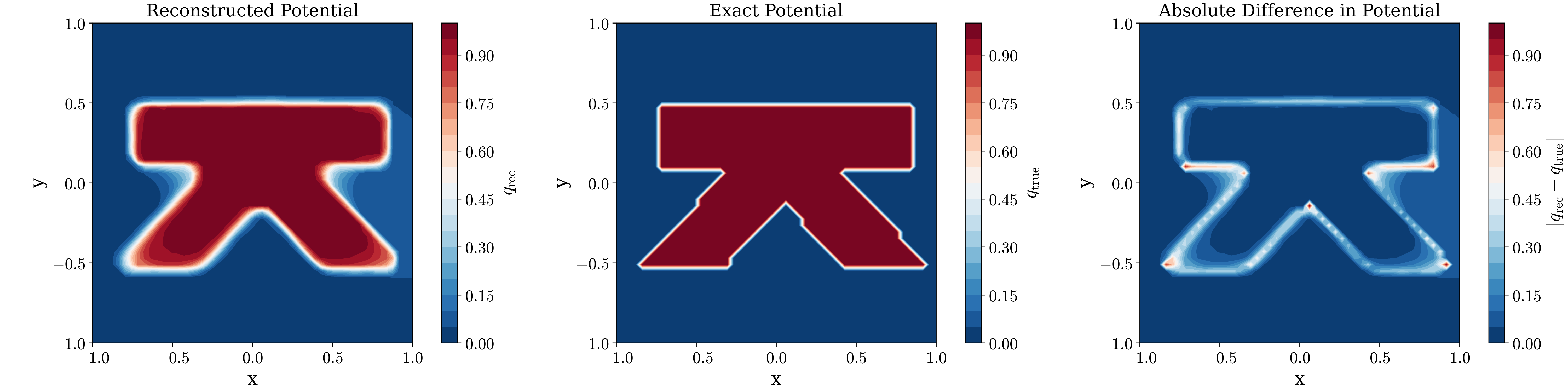}
  \caption{Example 3, fixed-energy scattering case with an oriented K-shaped potential: reconstruction (left), ground truth (middle), and residual (right). The TV regularization recovers the overall K-shape well, with larger errors near the sharp edges and an $L^2$ error of $0.212$.}
  \label{fig:k-shape}
\end{figure}

\begin{figure}[H]
  \centering
  \includegraphics[width=0.95\textwidth]{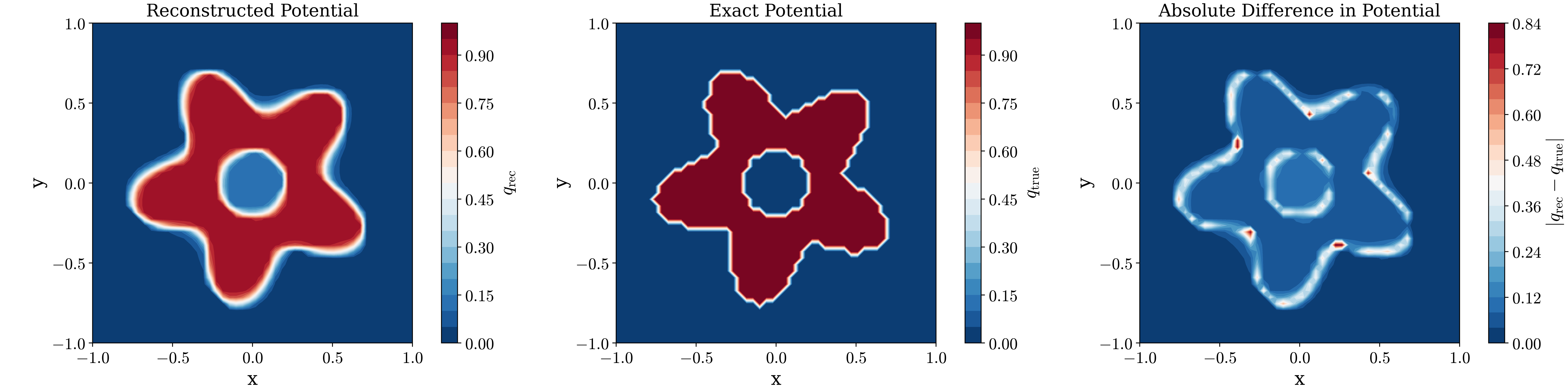}
  \caption{Example 4, fixed-energy scattering with a discontinuous flower-shaped potential: reconstruction (left), ground truth (middle), and residual (right). Both the location and amplitude are well recovered, with larger errors near the sharp edges, and an $L^2$ error of $0.213$.}
  \label{fig:flower-shape}
\end{figure}

\section*{Acknowledgements}

V.P., K.E-M., and T.T. received funding from the Finnish Ministry of Education and Culture through the Doctoral Programmes Pilot (Mathematics of Sensing, Imaging and Modelling project), the Research Council of Finland through the Advanced Mathematics for Sensing, Imaging and Modelling Flagship programme (grant 359186), and the Emil Aaltonen Foundation. T.L. was partly supported by the Research Council of Finland
(Centre of Excellence in Inverse Modelling and Imaging and FAME Flagship, grants 353091 and 359208).   M.L. was supported by the European Research Council under Advanced Grant 101097198, the Research Council of Finland Centre of Excellence, and the Research Council of Finland FAME Flagship programme (grant 359182). The views and opinions expressed in this paper are only those of the authors and do not necessarily represent those of the European Union or any of the funding bodies mentioned above. Neither the European Union nor any of the other funding organizations can be held responsible for them.

We would like to thank Andreas Hauptmann, Marko Huhtanen, and Marvin Kn\"oller for valuable 
discussions and suggestions that improved this work.

\bibliographystyle{abbrv}
\bibliography{nonlinear}

\noindent{\footnotesize E-mail addresses:\\
Khaoula El Maddah: khaoula.elmaddah@oulu.fi\\
Matti Lassas: matti.lassas@helsinki.fi\\
Tony Liimatainen: tony.liimatainen@helsinki.fi\\
Valter Pohjola: valter.pohjola@gmail.com\\
Teemu Tyni: teemu.tyni@oulu.fi
}

\end{document}